\documentclass[12pt]{amsart}
\usepackage[osf,sc]{mathpazo}
\usepackage{amssymb}

\usepackage{geometry}
\usepackage{bm}
\usepackage{graphicx}
\usepackage{hyperref}
\usepackage{varwidth}
\hypersetup{
	colorlinks=true, 
	linktoc=all,     
	linkcolor=blue,
	citecolor=red,
	filecolor=black,
	urlcolor=blue	
}
\usepackage{enumerate}
\usepackage[inline]{enumitem}
\makeatletter
\newcommand{\inlineitem}[1][]{%
	\ifnum\enit@type=\tw@
	{\descriptionlabel{#1}}
	\hspace{\labelsep}%
	\else
	\ifnum\enit@type=\z@
	\refstepcounter{\@listctr}\fi
	\quad\@itemlabel\hspace{\labelsep}%
	\fi} \makeatother
\newcommand{\ga}{\alpha}

\newcommand{\gd}{\delta}

\newcommand{\gs}{\sigma}

\newcommand{\gf}{\phi}

\newcommand{\Gd}{\Delta}

\newcommand{\Gom}{\Omega}

\newcommand{\subs}{\subset}
\newcommand{\sups}{\supset}

\newcommand{\bs}{\backslash}

\newcommand{\ti}{\tilde}

\newcommand{\mbb}{\mathbb}

\newcommand{\mcl}{\mathcal}

\newcommand{\us}{\underset}
\newcommand{\os}{\overset}

\newcommand{\lra}{\longrightarrow}
\newcommand{\llra}{\longleftrightarrow}

\newcommand{\La}{\Leftarrow}

\newcommand{\Ra}{\Rightarrow}
\newcommand{\Llra}{\Longleftrightarrow}

\newcommand{\es}{\emptyset}

\newcommand{\equ}[1]{%
	\begin{equation*}
		#1
	\end{equation*}
}
\newcommand{\equa}[1]{%
	\begin{equation*}
		\begin{aligned}
			#1
		\end{aligned}
	\end{equation*}
}
\newcommand{\equan}[2]{%
	\begin{equation}
		\label{Eq:#1}
		\begin{aligned}
			#2
		\end{aligned}
	\end{equation}
}

\DeclareMathOperator{\Ker}{Ker}
\DeclareMathOperator{\Ran}{Ran}

\newtheorem{theorem}{Theorem}[section]

\newtheorem{lemma}[theorem]{Lemma}

\newtheorem{ques}[theorem]{Question}
\newtheorem{claim}[theorem]{Claim}
\newtheorem{obs}[theorem]{Observation}

\theoremstyle{definition}
\newtheorem{defn}[theorem]{Definition}

\theoremstyle{remark}
\newtheorem{note}[theorem]{Note}
\newtheorem{remark}[theorem]{Remark}
\numberwithin{equation}{section}
\makeatletter
\def\namedlabel#1#2{\begingroup
	\def\@currentlabel{#2}%
	\label{#1}\endgroup
}
\makeatother
\newtheorem*{thmA}{\bf{Theorem A}}
\newtheorem*{thmB}{\bf{Theorem B}}

\begin{document}
\title[On Infinity Type Hyperplane Arrangements and Convex Positive Bijections]{On Infinity Type Hyperplane Arrangements and Convex Positive Bijections}
\author[C.P. Anil Kumar]{Author: C.P. Anil Kumar*}
\address{Center for Study of Science, Technology and Policy
\# 18 \& \#19, 10th Cross, Mayura Street,
Papanna Layout, Nagashettyhalli, RMV II Stage,
Bengaluru - 560094
Karnataka,INDIA}
\email{akcp1728@gmail.com}
\thanks{*The author is supported by a research grant and facilities provided by Center for study of Science, 
Technology and Policy (CSTEP), Bengaluru, INDIA for this research work.}
\begin{abstract}
In this article we prove in main Theorem~\ref{theorem:NRTMT} that any infinity type real hyperplane arrangement $\mcl{H}_n^m$ (Definition~\ref{defn:InfinityArrangements}) with the associated normal system $\mcl{N}$ (Definitions~[\ref{defn:NS},~\ref{defn:NSAHA}]) can be represented isomorphically (Definition~\ref{defn:Iso})
by another infinity type hyperplane arrangement $\ti{\mcl{H}}_n^m$ with a given associated normal system $\ti{\mcl{N}}$
if and only if the normal systems $\mcl{N}$ and $\ti{\mcl{N}}$ are isomorphic, that is, there is a convex positive bijection (Definition~\ref{defn:CPB}) 
between a pair of associated sets of normal antipodal pairs of vectors of $\mcl{N}$ and $\ti{\mcl{N}}$. 
\end{abstract}
\maketitle 
\section{\bf{Introduction and a brief survey}}


The theory of hyperplane arrangements is a well studied and vast subject.
There are many view points and perspectives on this subject. The literature survey
of this field consists of a lot of very good open problems. From a theoretical view point, 
A.~Dimca~\cite{MR3618796}, P.~Orlik \& H.~Terao~\cite{MR1217488} give an accessible introduction
to this subject to those who are interested in algebraic geometry and algebraic topology. In the context of arrangements,
matroids form combinatorial abstractions of vector configurations and hyperplane arrangements. E.~Katz~\cite{MR3702317}
gives a survey of theory of matroids aimed at algebraic geometers.
From a computational view point, problem nine in S.~Smale~\cite{MR1631413},~\cite{MR1754783}, is the following well known open question 
in this subject for the past two and a half decades. 

\begin{ques}
Does there exist a strongly polynomial time algorithm to decide the feasibility of the linear system of 
inequalities \equ{Ax \geq b} over the field of rational numbers?
\end{ques}
Also N.~Megiddo~\cite{MR1117042} mentions about the various computational aspects with a 
footnote mentioning the relevance of the subject to economists as well due to its vast applicability. 

Here in this article we consider the infinity type hyperplane arrangements (refer to Defintion~\ref{defn:InfinityArrangements}) over the field $\mbb{R}$ of real numbers and classify them by associating certain invariants such as normal systems (refer to Definition~\ref{defn:NS}). For infinity type line arrangements refer to C.~P.~Anil Kumar~\cite{CPAK}.  Here we answer when two real infinity type hyperplane arrangements in $\mbb{R}^m$ are isomorphic, though in principle, the field $\mbb{R}$ can be replaced by an ordered field $\mbb{F}$. 

\begin{remark}
	For those who are interested in generalising to ordered fields other than the subfields of $\mbb{R}$, I would like to mention that the basic theory of ordered fields is given in N.~Jacobson~\cite{MR0780184}~(Chapter $5$),\cite{MR1009787}~(Chapter $11$) and S.~Lang~\cite{MR1878556} (Chapter $11$).
\end{remark}

This type of field is also briefly mentioned in Survey~\cite{MR1117042} on page $229$ 
from an algebraic point of view. The association of 
invariants to hyperplane arrangements for the purpose of classification of the same is a well established 
method over various fields like $\mbb{Q},\mbb{R},\mbb{C}$ and finite fields $\mbb{F}_q$ where $q$ is a 
prime power. By associating invariants to hyperplane arrangements here, 
we give a criterion as to when an infinity type hyperplane arrangement is represented 
isomorphically by another such hyperplane arrangement with a given set of normals, more precisely, with a given normal system (refer to Definition~\ref{defn:NS}), 
where we prove main Theorem~\ref{theorem:NRTMT} or equivalently Theorem~\ref{theorem:SNRT}, both of which are stated in Section~\ref{sec:Preamble}. 
This result is new and its proof uses techniques from geometry of space and spatial arrangement of points in the field of linear algebra, convex geometry and theory of polytopes.

\section{\bf{Definitions and statement of the main result}}
\label{sec:Preamble}
We begin the section with a few definitions.
\begin{defn}[A Hyperplane Arrangement, A Generic Hyperplane Arrangement, A Central Hyperplane Arrangement]
	\label{defn:HA}
	~\\
	Let $m,n$ be positive integers. We say a set 
	\equ{\mcl{H}_n^m=\{H_1,H_2,\ldots,H_n\}} of $n$ affine hyperplanes in $\mbb{R}^m$ forms a 
	hyperplane arrangement. We say that they form a generic hyperplane arrangement or a hyperplane arrangement in general position, if Conditions 1,2 are satisfied.
	\begin{itemize}
		\item Condition 1: For $1\leq r \leq m$, the intersection of any $r$ hyperplanes has dimension $m-r$.
		\item Condition 2: For $r>m$, the intersection of any $r$ hyperplanes is empty.
	\end{itemize}
	We say that the hyperplane arrangement $\mcl{H}_n^m$ is central if $\us{i=1}{\os{n}{\cap}}H_i\neq \es$.
\end{defn}
 \begin{defn}[Normal System]
	\label{defn:NS}
	~\\
	Let $\mcl{N}=\{L_1,L_2,\ldots,L_n\}$ be a finite set of lines passing through the origin in $\mbb{R}^m$. Let $\mcl{U}=\{\pm v_1,\pm v_2,\ldots,\pm v_n\}$ be a set of antipodal pairs of non-zero vectors on these lines. We say that
	$\mcl{N}$ forms a normal system, if the set 
	\equ{\mcl{B}=\{v_1,v_2,\ldots,v_n\}}
	of vectors has the property that, any subset $\mcl{C}\subs \mcl{B}$ of cardinality at most $m$ is a linearly independent set. 
\end{defn}
\begin{defn}[Hyperplane arrangement given by a normal system]
	\label{defn:HAGivenbyNS}
	~\\
	Let $\mcl{N}=\{L_1,L_2,\ldots,L_n\}$ be a normal system in $\mbb{R}^m$. Let 
	$\mcl{U}=\{\pm (a_{i1},a_{i2},\ldots,a_{im})\mid 0\neq (a_{i1},a_{i2},\ldots,a_{im})\in L_i, 1\leq i\leq n\}$ 
	be a set of antipodal pairs of vectors of the normal system $\mcl{N}$. We fix the  matrix $[a_{ij}]_{1\leq i\leq n,1\leq j\leq m}\in M_{n\times m}(\mbb{R})$. Let $\mcl{H}_n^m=\{H_1,H_2,\ldots,H_n\}$ 
	be any hyperplane arrangement in $\mbb{R}^m$ whose equations are given by 
	\equ{H_i:\us{j=1}{\os{m}{\sum}} a_{ij}x_j=b_i \text{ for some }b_i\in \mbb{R}.}
	We say that the hyperplane arrangement $\mcl{H}_n^m$ is given by the normal system $\mcl{N}$.
\end{defn}
\begin{defn}[Normal System Associated to a Generic Hyperplane Arrangement]
	\label{defn:NSAHA}
	~\\
	Let $\mcl{H}^m_n=\{H_i:\us{j=1}{\os{m}{\sum}}a_{ij}x_j=b_i,1\leq i \leq n\}$ be a generic hyperplane arrangement. Then the normal system $\mcl{N}$
	associated to $\mcl{H}^m_n$ is given by
	\equ{\mcl{N}=\{L_i=\{t(a_{i1},a_{i2},\ldots,a_{im})\in \mbb{R}^m\mid t\in \mbb{R}\}\mid 1\leq i \leq n\}}
	and a set of antipodal pairs of normal vectors is given by 
	\equ{\mcl{U}=\{\pm v_1,\ldots,\pm v_n\}} where $0\neq v_i\in L_i, 1\leq i \leq n$.
	For example we can choose by default 
	\equ{\mcl{U}=\{\pm (a_{i1},a_{i2},\ldots,a_{im})\in \mbb{R}^m\mid 1\leq i \leq n\}.}
\end{defn}
\begin{defn}[Convex Positive Bijection and Isomorphic Normal Systems]
\label{defn:CPB}
~\\
Let  \equ{\mcl{N}_1=\{L_1,L_2,\ldots,L_n\},\mcl{N}_2=\{M_1,M_2,\ldots,M_n\}} 
be two finite sets of lines passing through the origin in $\mbb{R}^m$, both of them have the same 
cardinality $n$, which form normal systems. Let 
\equ{\mcl{U}_1=\{\pm v_1,\pm v_2,\ldots,\pm v_n\},\mcl{U}_2=\{\pm w_1,\pm w_2,\ldots,\pm w_n\}} 
be two sets of antipodal pairs of vectors on these lines in $\mcl{N}_1,\mcl{N}_2$ respectively. We say a bijection 
$\gd:\mcl{U}_1\lra \mcl{U}_2$ is a convex positive bijection if 
\equ{\gd(-u)=-\gd(u),u\in \mcl{U}_1} and for any basis $\mcl{B}=\{u_1,u_2,\ldots,u_m\}\subs \mcl{U}_1$
and a vector $u\in \mcl{U}_1$ we have 
\equa{u&=\us{i=1}{\os{m}{\sum}}a_iu_i \text{ with }a_i>0, 1\leq i \leq m, \text{ if and only if },\\
\gd(u)&=\us{i=1}{\os{m}{\sum}}b_i\gd(u_i) \text{ with }b_i>0, 1\leq i \leq m.}
We say two normal systems are isomorphic if there exists a convex positive bijection between their 
corresponding sets of antipodal pairs of vectors. 
\end{defn}
\begin{defn}[Isomorphism Between Two Generic Hyperplane Arrangements]
	\label{defn:Iso}
	Let
	\equ{(\mcl{H}_n^m)_1=\{H^1_1,H^1_2,\ldots,H^1_n\},
		(\mcl{H}_n^m)_2=\{H^2_1,H^2_2,\ldots,H^2_n\}}
	be two generic hyperplane arrangements in $\mbb{R}^m$. We say a map 
	$\gf:(\mcl{H}_n^m)_1 \lra (\mcl{H}_n^m)_2$
	is an isomorphism between these two generic hyperplane arrangements if $\gf$ is a bijection between the sets
	$(\mcl{H}_n^m)_1,(\mcl{H}_n^m)_2$, in particular on the subscripts $1\leq i\leq n$ satisfying the following property: given 
	$1\leq i_1<i_2<\ldots<i_{m-1}\leq n$ and lines \equ{L=H^1_{i_1}\cap H^1_{i_2}\cap \ldots \cap H^1_{i_{m-1}},
		M=H^2_{\gf(i_1)}\cap H^2_{\gf(i_2)}\cap \ldots \cap H^2_{\gf(i_{m-1})},}
	the order of vertices, that is, zero-dimensional intersections on the lines $L,M$, agrees via the bijection induced by $\gf$ again 
	on the sets of subscripts of cardinality $m$ (corresponding to the vertices on $L$) containing 
	$\{i_1,i_2,\ldots,i_{m-1}\}$ and those (corresponding to the vertices on $M$) containing 
	$\{\gf(i_1),\gf(i_2),\ldots,\gf(i_{m-1})\}$. There are four possibilities of pairs of orders and any one
	pairing of orders out of these four pairs must agree via the map induced by $\gf$. We say the isomorphism $\gf$ preserves subscripts or $\gf$ is trivial on subscripts if in addition to being an isomorphism it satisfies $\gf(H_i^1)=H_i^2$ for $1\leq i\leq n$.
\end{defn}
\begin{note}
	If there is an isomorphism between two generic hyperplane arrangements $(\mcl{H}_n^m)_i,i=1,2$,
	then there exists a piecewise linear bijection of $\mbb{R}^m$ to $\mbb{R}^m$ which takes one arrangement to 
	another, using suitable triangulation of the polyhedral regions. For obtaining a piecewise linear isomorphism 
	extension from vertices to the one-dimensional skeleton of the arrangements, further subdivision is not needed.
\end{note}
\begin{defn}[Central Point on a Line]
Let $L$ be a line in an Euclidean space and let $P_1,P_2,P_3$ be three points on the line $L$. We say $P_2$ is the central point among $P_1,P_2,P_3$ if the order of these points on the line $L$ is either $P_1,P_2,P_3$ or $P_3,P_2,P_1$. 
\end{defn}
Since the main of theme of this article is based on central points, here we mention a theorem on preservation of central points without proof as this is a standard theorem.
\begin{theorem}[A Theorem on Preservation of Central Points]
\label{theorem:PCP}
~\\
Let 
\equ{(\mcl{H}_n^m)_1=\{H^1_1,H^1_2,\ldots,H^1_n\},
(\mcl{H}_n^m)_2=\{H^2_1,H^2_2,\ldots,H^2_n\}}
be two generic hyperplane arrangements in $\mbb{R}^m$. Let $\gf:(\mcl{H}_n^m)_1 \lra (\mcl{H}_n^m)_2$
be a bijection. Using the subscripts let $\gf^{vert}:Vert((\mcl{H}_n^m)_1) \lra 
Vert((\mcl{H}_n^m)_2)$ be the induced map on the vertices of the arrangements, that is, on the subsets of $\{1,2,\ldots,n\}$ of cardinality $m+1$. 
Then $\gf$ is an isomorphism of the arrangements if and only if for any three vertices 
on any line of the arrangement in $(\mcl{H}_n^m)_1$ the map $\gf^{vert}$ preserves the centrality property for the central vertex.
\end{theorem}
\begin{defn}[An Hyperplane at Infinity]
	\label{defn:HAInfinity}
	Let $\mcl{H}^m_n$ be a generic hyperplane arrangement in $\mbb{R}^m$.
	We say a hyperplane $H \subs \mbb{R}^m$ is a hyperplane at infinity with respect to
	$\mcl{H}^m_n$,  if 
	\begin{enumerate}
		\item $\mcl{H}^m_n \cup\{H\}$ is a generic hyperplane arrangement and
		\item 	all the bounded intersections of the arrangement $\mcl{H}^m_n$, 
		that is, the zero-dimensional vertices of intersection of the arrangement $\mcl{H}^m_n$ lie only on ``one side" of $H$ (possibly the ``one side" includes the hyperplane $H$ if $H\in \mcl{H}^m_n$).
	\end{enumerate}
\end{defn}
\begin{defn}[An Infinity Type Arrangement]
	\label{defn:InfinityArrangements}
	~\\
	Let $\mcl{H}^n_m=\{H_1,H_2,\ldots,H_n\}$ be a generic hyperplane arrangement in $\mbb{R}^m$.
	We say $\mcl{H}^n_m$ is an infinity type arrangement if there exists a permutation $\gs\in S_n$ such that
	the hyperplane $H_{\gs(l)}$ is a hyperplane at infinity with respect to the arrangement 
	\equ{\{H_{\gs(1)},H_{\gs(2)},\ldots,H_{\gs(l-1)}\},\ 2\leq l\leq n.}
	The permutation $\gs$ is called an infinity permutation for the arrangement $\mcl{H}^n_m$. It need not be unique.
\end{defn}
\begin{remark}
The infinity type line arrangements, their nomenclature and some of their properties are discussed in C.~P.~Anil Kumar~\cite{CPAK}. 
\end{remark}
We state the main theorem of this article and also state another equivalent or more symmetric one.

\begin{thmA}[Normal Representation Theorem: Main Theorem]
\namedlabel{theorem:NRTMT}{A}
~\\
Let $\mcl{N}_1=\{L_1,L_2,\ldots,L_n\}$ be a normal system in $\mbb{R}^m$ of cardinality $n$ and $\mcl{U}_1$ be a set of 
antipodal pairs of vectors of this normal system. Let $\mcl{H}_n^m$ be an infinity type hyperplane 
arrangement in $\mbb{R}^m$ and $\mcl{N}_2$ be the normal system associated to $\mcl{H}_n^m$ with $\mcl{U}_2$ a set of antipodal pairs of normal vectors of the normal system $\mcl{N}_2$. Then the infinity type hyperplane arrangement $\mcl{H}_n^m$ with normal system $\mcl{N}_2$ can be represented 
isomorphically by another infinity type hyperplane arrangement with normal system $\mcl{N}_1$ if and only if there exists a convex positive bijection between $\mcl{U}_1$ and $\mcl{U}_2$.
\end{thmA}

\begin{thmB}[Normal Representation Theorem: Symmetric Form]
\namedlabel{theorem:SNRT}{B}
~\\
If two infinity type hyperplane arrangements $(\mcl{H}_n^m)_1$, $(\mcl{H}_n^m)_2$ in $\mbb{R}^m$
are isomorphic then their associated normal systems $\mcl{N}_1$ and $\mcl{N}_2$ are isomorphic. 
Conversely, if we have two infinity type hyperplane arrangements $(\mcl{H}_n^m)_1$, $(\mcl{H}_n^m)_2$  in $\mbb{R}^m$,
whose associated normal systems $\mcl{N}_1$ and $\mcl{N}_2$ are isomorphic, then, there exist translates of 
each of the hyperplanes in the hyperplane arrangement $(\mcl{H}_n^m)_2$, giving rise to a translated infinity type
hyperplane arrangement $\ti{\mcl{H}}_n^m$, such that, $\ti{\mcl{H}}_n^m$ and $(\mcl{H}_n^m)_1$
are isomorphic.
\end{thmB}
\begin{remark}
The statement of Theorems~\ref{theorem:NRTMT},~\ref{theorem:SNRT} for $m=1$, that is, for infinity type line arrangements in the plane, is clear and immediate.
\end{remark}
Now we mention an important observation which is the crux in proving Theorem~\ref{theorem:SNRT} later.
This observation is a straight forward observation in the plane $\mbb{R}^2$.
\begin{figure}[h]
\centering
\includegraphics[width = 1.0\textwidth]{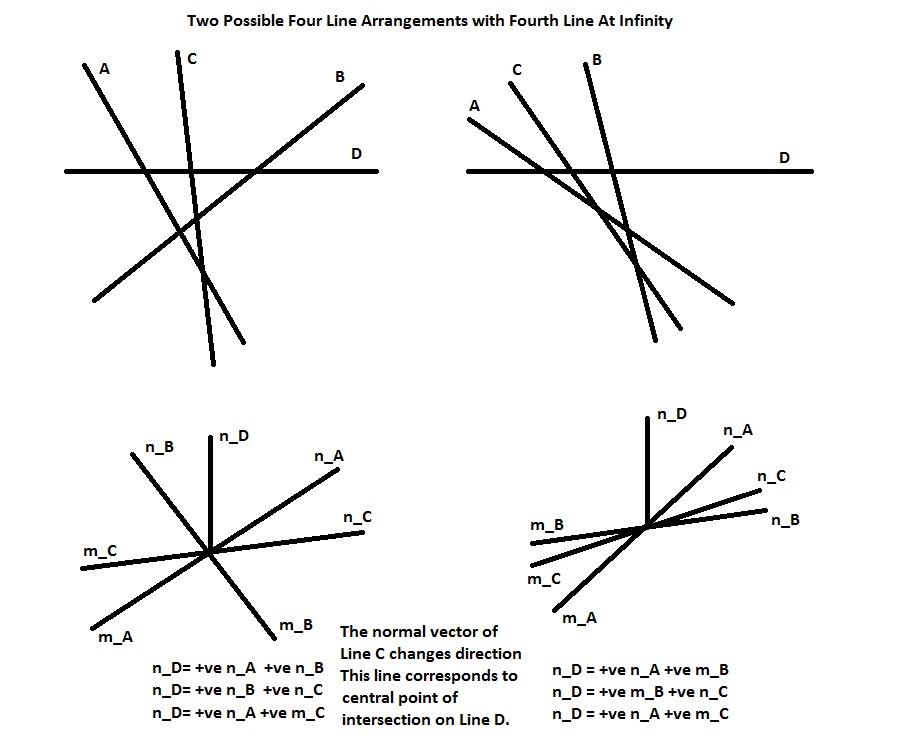}
\caption{Central Point and Normal Directions}
\label{fig:MinusOne}
\end{figure}
\begin{obs}[On the Central Point of Intersection]
\label{obs:ChangeSign}
~\\
\begin{center}
\fbox{\begin{varwidth}{\dimexpr\textwidth-2\fboxsep-2\fboxrule\relax}
If we have three pairwise intersecting generic lines, intersecting a fourth generic line at infinity in a plane ~\emph{(}refer to Definition~\emph{\ref{defn:HAInfinity})}, and
if we express a normal of the fourth line as a positive linear combination of any pair of normals of the  
other three lines then the normal direction of the line corresponding to only the central point of
intersection on the fourth line reverses its sign/direction ~\emph{(}refer to Figure~~\emph{\ref{fig:MinusOne})}.
\end{varwidth}}
\end{center}
\end{obs}

\section{\bf{Hyperplanes at infinity and an extension theorem}}
\label{sec:HIExt}
We prove below an extension theorem for an isomorphism between two isomorphic generic hyperplane arrangements 
which allows us to extend isomorphisms when we add hyperplanes at infinity under certain conditions. 
Theorem~\ref{theorem:ExtTheorem} is used later in proving Theorem~\ref{theorem:SNRT}. 

\begin{note}
Given a normal direction, there exist two parallel 
hyperplanes at infinity, with the given normal direction, on either side of the bounded intersections of the hyperplane arrangement that can possibly be added as a hyperplane at either one of the infinities to the hyperplane arrangement.
\end{note}

First we state a lemma before stating Theorem~\ref{theorem:ExtTheorem}.
\begin{lemma}
\label{lemma:ExtTheoremBaseCase}
Let $\mcl{L}^2_3=\{L_1,L_2,L_3\},\ti{\mcl{L}}^2_3=\{\ti{L}_1,\ti{L}_2,$ $\ti{L}_3\}$ be two generic line arrangements, (that is, sets of three generic lines) in the plane $\mbb{R}^2$.
Let $L_4,L_4'$ be two parallel lines at infinity in $\mbb{R}^2$ on either side of the bounded set of points of intersection $\{L_1\cap L_2,L_2\cap L_3,L_1\cap L_3\}$, giving rise to generic line arrangements 
$\mcl{L}^2_4=\{L_1,L_2,L_3,L_4\}$ and $\mcl{L}^{2'}_4=\{L_1,L_2,L_3,L_4'\}$ respectively. Let $\ti{L}_4,\ti{L}_4'$ be two parallel lines at infinity in $\mbb{R}^2$ on either side of the bounded set of points of
intersection $\{\ti{L}_1\cap \ti{L}_2,\ti{L}_2\cap \ti{L}_3,\ti{L}_1\cap \ti{L}_3\}$, giving rise to generic line arrangements $\ti{\mcl{L}}^{2}_4=\{\ti{L}_1,\ti{L}_2,\ti{L}_3,\ti{L}_4\}$ and 
$\ti{\mcl{L}}^{2'}_4=\{\ti{L}_1,\ti{L}_2,\ti{L}_3,\ti{L}_4'\}$ respectively.
Then 
\begin{enumerate}
\item For any $1\leq i\leq 3$, $L_i\cap L_4$ is the central point of intersection on $L_4$ among the three points $L_1\cap L_4,L_2\cap L_4,L_3\cap L_4$ if and only if 
$L_i\cap L_4'$ is the central point of intersection on $L_4'$ among the three points $L_1\cap L_4',L_2\cap L_4',L_3\cap L_4'$. 
\item For any $1\leq i\leq 3$, $\ti{L}_i\cap \ti{L}_4$ is the central point of intersection on $\ti{L}_4$ among the three points $\ti{L}_1\cap \ti{L}_4,\ti{L}_2\cap \ti{L}_4,\ti{L}_3\cap \ti{L}_4$ if and only if 
$\ti{L}_i\cap \ti{L}_4'$ is the central point of intersection on $\ti{L}_4'$ among the three points $\ti{L}_1\cap \ti{L}_4',\ti{L}_2\cap \ti{L}_4',\ti{L}_3\cap \ti{L}_4'$. 
\begin{figure}[h]
	\centering
	\includegraphics[width = 1.0\textwidth]{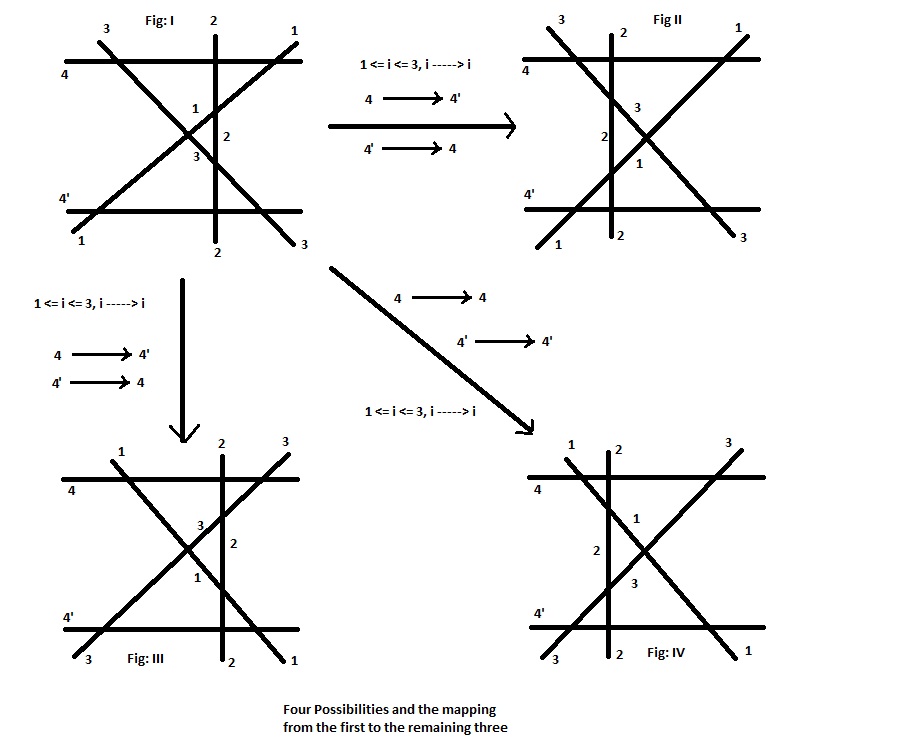}
	\caption{The Base Case Illustration of Extension Theorem}
	\label{fig:Zero}
\end{figure}
\item Suppose for some particular $1\leq j\leq 3$ we have $L_j\cap L_4$ is the central point of intersection on $L_4$ among the three points $L_1\cap L_4,L_2\cap L_4,L_3\cap L_4$
and also $\ti{L}_j\cap \ti{L}_4$ is the central point of intersection on $\ti{L}_4$ among the three points $\ti{L}_1\cap L_4,\ti{L}_2\cap \ti{L}_4,\ti{L}_3\cap \ti{L}_4$. Then we have 
\begin{itemize}
\item either $\mcl{L}^2_4 \os{\gf}{\cong} \ti{\mcl{L}}^2_4$ and $\mcl{L}^{2'}_4 \os{\psi}{\cong} \ti{\mcl{L}}^{2'}_4$ by isomorphisms $\gf,\psi$ preserving the subscripts $\{1,2,3,4\}$ of the lines in the arrangements,
\item or $\mcl{L}^2_4 \os{\gf}{\cong} \ti{\mcl{L}}^{2'}_4$ and $\mcl{L}^{2'}_4 \os{\psi}{\cong} \ti{\mcl{L}}^{2}_4$ by isomorphisms $\gf,\psi$ preserving the subscripts $\{1,2,3,4\}$ of the lines in the arrangements.
\end{itemize}
or equivalently under the isomorphisms, for $1\leq i\leq 3$, $(L_i\os{\gf,\psi}{\llra} \ti{L}_i)$ pairings occur and 
\begin{itemize}
\item either $(L_4 \os{\gf}{\llra} \ti{L}_4)$ and $(L_4' \os{\psi}{\llra} \ti{L}_4')$ pairings occur,
\item or $(L_4 \os{\gf}{\llra} \ti{L}_4')$ and $(L_4' \os{\psi}{\llra} \ti{L}_4)$ pairings occur.
\end{itemize}
\end{enumerate}
\end{lemma}

\begin{proof}
The proof of the lemma is an immediate observation about central points on all the four lines of each of the line arrangements $\mcl{L}^2_4,\mcl{L}^{2'}_4,\ti{\mcl{L}}^2_4,\ti{\mcl{L}}^{2'}_4$.
We illustrate the proof of the lemma with Figure~\ref{fig:Zero}. In this figure Fig:I can be taken as the line arrangements $\mcl{L}^2_4,\mcl{L}^{2'}_4$ and any one of Fig:II, Fig:III, Fig:IV can be taken as 
line arrangements $\ti{\mcl{L}}^{2}_4,\ti{\mcl{L}}^{2'}_4$. Finally we obtain isomorphisms as 
\equa{\text{either }&4 \llra 4, 4'\llra 4', i\llra i, 1\leq i\leq 3, Fig:IV\\
\text{or }&4 \llra 4', 4' \llra 4, i\llra i, 1\leq i\leq 3, Fig:II, Fig:III} 
The possibilities can be taken as the base case for extension Theorem~\ref{theorem:ExtTheorem}.
\end{proof} 
Here we state and prove extension Theorem~\ref{theorem:ExtTheorem}.
\begin{theorem}[Extension Theorem]
\label{theorem:ExtTheorem}
~\\
Let $m>1,n>m+1$ be two positive integers. Let 
\equ{\mcl{H}^m_{n-1}=\{H_1,H_2,\ldots,H_{n-1}\},\ti{\mcl{H}}^m_{n-1}=
\{\ti{H}_1,\ti{H}_2,\ldots,\ti{H}_{n-1}\}} be two isomorphic generic hyperplane arrangements by an isomorphism $\gf$
which takes $H_i \lra \ti{H}_i,1\leq i \leq n-1$. Now suppose 
\equ{H_n^1,H_n^2 \text{ and } \ti{H}_n^1, \ti{H}_n^2}
are two pairs of parallel hyperplanes at either infinities with respect to the arrangements
$\mcl{H}^m_{n-1}$ and $\ti{\mcl{H}}^m_{n-1}$. (Hence all the bounded intersections 
lie in the hyperplane strip spaces which is in between the two pairs of parallel hyperplanes.) 
If the induced map $\gf_{\mid_{induced}}$ of the isomorphism $\gf$ on the following two hyperplane arrangements  
\equa{^r\mcl{M}^{m-1}_{n-1}&=\{H^r_n\cap H_1,H^r_n\cap H_2,\ldots,H^r_n\cap H_{n-1}\},\\
^s\ti{\mcl{M}}^{m-1}_{n-1}&=\{\ti{H}^s_n\cap \ti{H}_1,\ti{H}^s_n\cap \ti{H}_2,\ldots,\ti{H}^s_n\cap \ti{H}_{n-1}\}}
is an isomorphism for any one $(r,s)\in \{1,2\}\times \{1,2\}$
then the isomorphism $\gf$ extends to an isomorphism  
$\ti{\gf}$ on the following two generic hyperplane arrangements 
\equ{\mcl{H}^m_{n-1}\cup \{H^{r_1}_n\},\ti{\mcl{H}}^m_{n-1}\cup \{\ti{H}^{s_1}_n\}}
which takes \equ{H_i\llra \ti{H}_i,1\leq i \leq n-1,H^{r_1}_n \llra \ti{H}^{s_1}_n}
for some choice of $(r_1,s_1)\in \{1,2\}\times \{1,2\}$. Moreover the isomorphism also extends to an isomorphism
for the complementary ordered pair $(r_2,s_2)$ of $(r_1,s_1)$ where $\{r_1,r_2\}=\{s_1,s_2\}=\{1,2\}$.
\end{theorem}
\begin{proof}

Using the fact that the induced map $\gf_{\mid_{induced}}$ is an isomorphism and 
the original map $\gf$ is an isomorphism we prove that the order of intersections agree on all the one -dimensional lines for some choice of $(r_1,s_1)\in \{1,2\}\times \{1,2\}$ in the hyperplane arrangements 
\equ{\mcl{H}^m_{n-1}\cup \{H^{r_1}_n\},\ti{\mcl{H}}^m_{n-1}\cup \{\ti{H}^{s_1}_n\}}
under the map $\ti{\gf}$ for a suitable definition of $\ti{\gf}$. Now we make an important observation. 

To find an extension $\ti{\gf}$ and to prove this theorem we restrict our space of attention to two-dimensional planes 
of interest as follows.
Let \equ{1\leq i_1<i_2<\ldots<i_{m-2}\leq n-1<n, 1\leq i_{m-1}<i_m<i_{m+1}\leq n-1<n} 
and 
\equ{\{i_{m-1}<i_m<i_{m+1}\}\cap \{i_1<i_2<\ldots<i_{m-2}\}=\es.} 

Let  
\equa{L_{\{i_1,i_2,\ldots,i_{m-2},i_{m-1}\}}&=H_{i_1}\cap \ldots \cap H_{i_{m-2}}\cap H_{i_{m-1}}\\
L_{\{i_1,i_2,\ldots,i_{m-2},i_{m}\}}&=H_{i_1}\cap \ldots \cap H_{i_{m-2}}\cap H_{i_{m}}\\
L_{\{i_1,i_2,\ldots,i_{m-2},i_{m+1}\}}&=H_{i_1}\cap \ldots \cap H_{i_{m-2}}\cap H_{i_{m+1}}}

\equa{\ti{L}_{\{i_1,i_2,\ldots,i_{m-2},i_{m-1}\}}&=\ti{H}_{i_1}\cap \ldots \cap \ti{H}_{i_{m-2}}\cap \ti{H}_{i_{m-1}}\\
\ti{L}_{\{i_1,i_2,\ldots,i_{m-2},i_{m}\}}&=\ti{H}_{i_1}\cap \ldots \cap \ti{H}_{i_{m-2}}\cap \ti{H}_{i_{m}}\\
\ti{L}_{\{i_1,i_2,\ldots,i_{m-2},i_{m+1}\}}&=\ti{H}_{i_1}\cap \ldots \cap \ti{H}_{i_{m-2}}\cap \ti{H}_{i_{m+1}}}

be the corresponding triples of lines with the corresponding two-dimensional planes of interest being
\equ{H_{i_1}\cap H_{i_2} \cap \ldots \cap H_{i_{m-2}}, \ti{H}_{i_1}\cap \ti{H}_{i_2} \cap \ldots \cap \ti{H}_{i_{m-2}}}

For $(r,s)\in \{1,2\}\times \{1,2\}$ the pairs of parallel hyperplanes at infinity gives rise to lines at infinity 
in the two-dimensional planes of interest respectively. The lines are given by
\equa{L^r_{\{i_1,i_2,\ldots,i_{m-2},n\}}&=H_{i_1}\cap H_{i_2} \cap \ldots \cap H_{i_{m-2}} \cap H^r_n, r\in \{1,2\}\\
\ti{L}^s_{\{i_1,i_2,\ldots,i_{m-2},n\}}&=\ti{H}_{i_1}\cap \ti{H}_{i_2} \cap \ldots \cap \ti{H}_{i_{m-2}}
\cap \ti{H}^s_n, s\in \{1,2\}}

\begin{figure}[h]
\centering
\includegraphics[width = 1.0\textwidth]{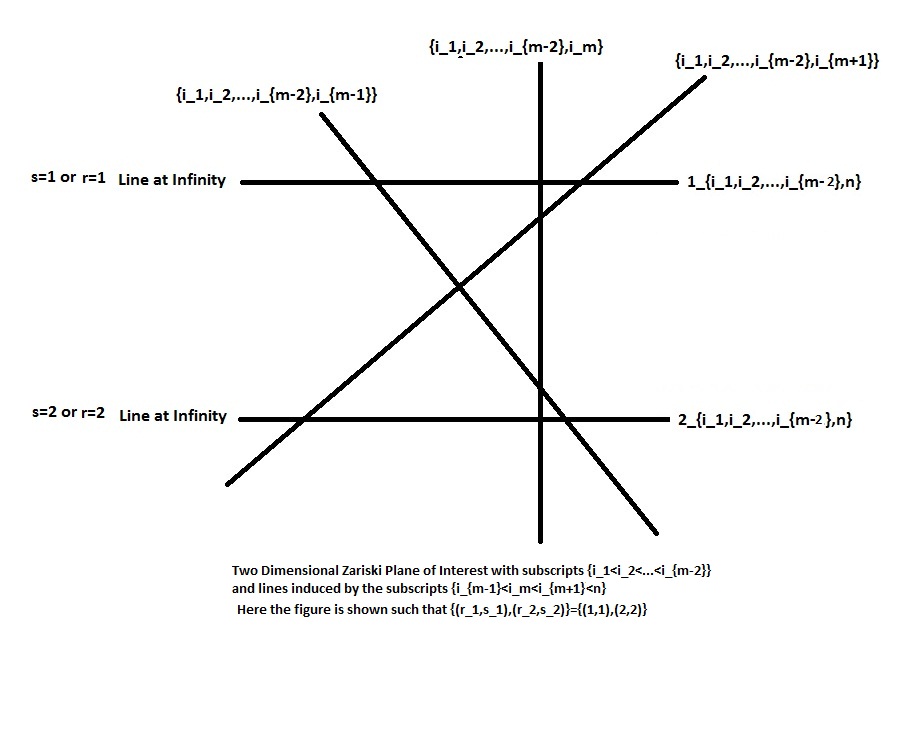}
\caption{Two-dimensional zariski plane of interest $\{i_1<i_2<\ldots<i_{m-2}\}$}
\label{fig:One}
\end{figure}

Now consider Figure~\ref{fig:One}. As we have both isomorphisms $\gf$ for points of the lines not on the new hyperplanes and $\gf_{induced}$ for points on the lines of the new hyperplanes $H_n,\ti{H}_n$
we can use Lemma~\ref{lemma:ExtTheoremBaseCase} on each of the two-dimensional planes of interest to obtain an isomorphic pairing $(r_1,s_1)\in \{1,2\}\times \{1,2\}$ 
and its complementary pair $(r_2,s_2) \in \{1,2\}\times \{1,2\}$. This can be obtained coherently depending on $\{i_1<\ldots<i_{m-2}\}$, that is, the two-dimensional zariski plane of interest and independent of 
the subscripts $\{i_{m-1}<i_m<i_{m+1}\}$.

Now the only possible ambiguity is whether the pair $(r_1,s_1)$ and the complementary pair $(r_2,s_2)$ is the same for all two-dimensional planes of interest. There are two possible choices.
$\{(r_1,s_1),(r_2,s_2)\}=\{(1,1),(2,2)\}$ or $\{(r_1,s_1),(r_2,s_2)\}=\{(1,2),(2,1)\}$.
To remove this ambiguity we use continuity and connectedness arguments. We use usual topology over the field of reals. Otherwise we 
use zariski topology over infinite fields in particular over an ordered field.

We obtain the same pair $(r_1,s_1)$ and its complementary pair $(r_2,s_2)$ for all two-dimensional planes of interest, because, the map which takes union of
zariksi planes of interest, the skeleton of two-dimensional planes of the hyperplane arrangement to the set with discrete topology containing two 
elements which are complementary extension pairs \equ{\{\{(1,1),(2,2)\},\{(1,2),(2,1)\}\}} is continuous. This is because
\begin{itemize}
\item it is easy for the reader to note that the map agrees (patches up) on the intersection of two such planes (if they intersect) which is a line of the 
arrangement, since the isomorphic pairing $(r_1,s_1),(r_2,s_2)$ is the same for all two-dimensional zariski planes of interest $\{i_1<i_2<\ldots<i_{m-2}\}$
which contains a fixed line of the arrangement say $\{j_1<j_2<\ldots<j_{m-1}\} \sups \{i_1<i_2<\ldots<i_{m-2}\}$.

\item Inverse image of a single point is a finite union of two-dimensional zariski planes and hence it is closed.
\item Moreover the union of planes of 
interest of the arrangement is connected (point set topological sense) in zariski topology (refer to 
Theorem~\ref{theorem:ZCHA} with $k=2$).
\end{itemize}
In Figure~\ref{fig:One} for any choice of plane of 
interest we have either $(r_1,s_1)=(1,1)$ or $(r_1,s_1)=(2,2)$ and not $(1,2),(2,1)$. This proves exactly the statement 
of the theorem.
\end{proof}

\section{\bf{Existence of orthogonal projections over ordered fields}}
\label{sec:EOP}
~\\
This section is for those readers who are interested in the results of this article over ordered fields other than $\mbb{R}$. This section can be skipped otherwise.
Here we prove the existence of certain projections which will be useful to the proof of the main theorem in the next section.
We note orthogonal projections exist over ordered fields even though square roots of a general positive element need not be in the field.
Let $\mbb{F}$ be an ordered field. Let $v^i=(x_1^i,x_2^i,\ldots,x_n^i)^t, 1\leq i\leq k$ be any finite set of 
linearly independent vectors in $\mbb{F}^n$ for $k\leq n$ spanning a given subspace. 
Define a linear transformation $T$ given as follows. 
\equ{T:\mbb{F}^n\lra \mbb{F}^k \text{ where }[T]_{k\times n}=[x_i^j]_{1\leq i\leq n,1\leq j\leq k}}
Now we have $\Ker(T)=<v^i:1\leq i\leq k>^{\perp}$. We have row rank of $T$ is $k$. Since 
\equ{row\operatorname{-}rank(T)=col\operatorname{-}rank(T), Rank + Nullity = n}
we have $\dim(\Ker(T))=n-k$.
Define on $\mbb{F}^m$ with $m>0$ a positive integer,
\equ{<v=(x_1,x_2,\ldots,x_m),w=(y_1,y_2,\ldots,y_m)>_{\mbb{F}^m}=\us{i=1}{\os{m}{\sum}}x_iy_i.}
This is a symmetric bilinear form with the property that 
\begin{itemize}
\item $<v,v>_{\mbb{F}^m}\  \geq 0$ for $v\in \mbb{F}^m$.
\item $<v,v>_{\mbb{F}^m}=0 \Llra v=0$.
\end{itemize}
Then for $w_1\in \mbb{F}^n,w_2\in \mbb{F}^k$
\equ{<Tw_1,w_2>_{\mbb{F}^k}=w_2^tTw_1=w_1^tT^tw_2=<w_1,T^tw_2>_{\mbb{F}^n}.}
Now we observe that if $w_1\in \Ker(T) \Llra <w_1,T^tw_2>_{\mbb{F}^n}=0$ for all $w_2\in \mbb{F}^k$. 
So we conclude that 
\equ{\Ker(T)^{\perp}=\Ran(T^t)=Span<v^i:1\leq i\leq k>.}
So we conclude that 
\equ{\Ker(T)\bigoplus \Ran(T^t)=\mbb{F}^n.}
We define the orthogonal projections as $P,Q:\mbb{F}^n\lra \mbb{F}^n$ such that \equ{P_{\mid_{\Ker(T)}}=0,
P_{\mid_{\Ran(T^t)}}=Id,Q_{\mid_{\Ker(T)}}=Id,Q_{\mid_{\Ran(T^t)}}=0.}
These projections satisfy the following relations. \equ{I=P+Q,P^2=P,Q^2=Q,P^t=P,Q^t=Q, \text{ that is}}
\equ{<Pw_1,w_2>_{\mbb{F}^n}=<w_1,Pw_2>_{\mbb{F}^n},<Qw_1,w_2>_{\mbb{F}^n}=<w_1,Qw_2>_{\mbb{F}^n}}
for $w_1,w_2\in \mbb{F}^n$. This proves the existence of orthogonal projections.

\section{\bf{Proof of the main theorem}}
\label{sec:PMTA}
In this section we prove main Theorem~\ref{theorem:NRTMT} by proving Theorem~\ref{theorem:SNRT}.
Here we prove both the following implications.
\equ{\text{Normal Systems are isomorphic }\Llra \text{ Theorem~\ref{theorem:SNRT} holds.}}
\begin{proof}[Proof $\Ra$]
We assume that the infinity permutation of the infinity hyperplane arrangement $(\mcl{H}^n_m)_1$ is identity by renumbering the subscripts.
Suppose the normal systems $\mcl{U}_1,\mcl{U}_2$ are isomorphic. Let $\gd:\mcl{U}_1 \lra \mcl{U}_2$ 
be a convex positive bijection. Without loss generality let us assume that $\gd$ induces trivial permutation, 
that is, identity on subscripts by renumbering the subscripts of the arrangement $(\mcl{H}^n_m)_2$ if necessary. Suppose, using the bijection, we have constructed isomorphic infinity type (generic) hyperplane arrangements, with identity infinity permutations (refer to Definition~\ref{defn:InfinityArrangements}), given by 
\equ{(\mcl{H}^{m}_{l-1})_1=\{H^1_1,H^1_2,\ldots,H^1_{l-1}\},
(\mcl{H}^{m}_{l-1})_2=\{H^2_1,H^2_2,\ldots,H^2_{l-1}\},}
for $l>m+1$. Note that for $l-1\leq m+1$ these two hyperplane arrangements are isomorphic by an 
isomorphism which is identity on the subscripts. Then we add, hyperplanes at infinity $H^1_l$ and $H^2_l$,
whose subscripts correspond to each other under the bijection $\gd$, to the arrangements 
$(\mcl{H}^{m}_{l-1})_1,(\mcl{H}^{m}_{l-1})_2$ respectively.
Now we prove the following. The induced generic hyperplane arrangements
\equa{(\mcl{M}^{m}_{l-1})_1&=\{H^1_1\cap H^1_l,H^1_2\cap H^1_l,\ldots,H^1_{l-1}\cap H^1_l\}\\
(\mcl{M}^{m}_{l-1})_2&=\{H^2_1\cap H^2_l,H^2_2\cap H^2_l,\ldots,H^2_{l-1}\cap H^2_l\}}
are isomorphic again by an isomorphism which is identity on the subscripts.

Before we prove this we define the following. First we observe that for $i=1,2$ any zero-dimensional vertex on a line
of the arrangement in the hyperplane $H^i_l$ is an intersection of a line of the arrangement $(\mcl{M}^{m}_{l-1})_i$ which is contained in $H^i_l$ and a line of the arrangement $(\mcl{H}^{m}_{l-1})_i$
which is not contained in the hyperplane $H^i_l$. For every line 
\equ{H^i_{k_1}\cap H^i_{k_2} \cap \ldots \cap H^i_{k_{m-1}},1\leq k_1<k_2<\ldots<k_{m-1}\leq l-1}
not in the hyperplane $H^i_l$ we associate by choosing a direction vector 
\equ{n^i_{\{k_1,k_2,\ldots,k_{m-1}\}}}
which is outward pointing on the other side of the bounded intersections of $(\mcl{H}^{m}_{l-1})_i$ and 
which makes a positive dot product with an outward normal of $H^i_l$ which is also on the other side. 
The positive dot product is obtained by evaluating the linear functional of the outward normal
of $H^i_l$ at normals $n^i_{\{k_1,k_2,\ldots,k_{m-1}\}}$. 
Now we prove the following claim.
\begin{claim}
If there exists a convex positive bijection $\gd:\mcl{U}_1\lra \mcl{U}_2$ which is identity on the subscripts then 
the map of the directions 
\equ{\{n^1_{\{k_1,k_2,\ldots,k_{m-1}\}}\mid 1\leq k_1<k_2<\ldots<k_{m-1}\leq l-1\}} to the directions 
\equ{\{n^2_{\{k_1,k_2,\ldots,k_{m-1}\}}\mid 1\leq k_1<k_2<\ldots<k_{m-1}\leq l-1\}}
taking 
\equ{n^1_{\{k_1,k_2,\ldots,k_{m-1}\}} \lra n^2_{\{k_1,k_2,\ldots,k_{m-1}\}}}
with the respective subscript satisfies the convexity triple property for the lines, that is,
if $A,B,C$ denote three subsets of $\{1,2,\ldots,l-1\}$ each of cardinality $m-1$ given by 
\equa{A&=\{j_1<j_2<\ldots<j_{m-2}\}\cup\{j_{m-1}\},\\
C&=\{j_1<j_2<\ldots<j_{m-2}\}\cup\{j_m\},\\
B&=\{j_1<j_2<\ldots<j_{m-2}\}\cup \{j_{m+1}\}} 
then \equa{&n^1_{C}=a_1 n^1_{A} + b_1 n^1_{B} \text{ for some } a_1>0,b_1>0\\
&\Llra n^2_{C}=a_2 n^2_{A} + b_2 n^2_{B} \text{ for some } a_2>0,b_2>0.}
\end{claim}
\begin{proof}[Proof of Claim]
We observe that for $i=1,2$ the vectors $n^i_A,n^i_B,n^i_C$ span a two-dimensional space. 
Hence there exist coefficients $a_i,b_i\in \mbb{R}^{*}=\mbb{R}\bs \{0\}$ such that 
$a_in^i_A+b_in^i_B=n^i_C$. For $i=1,2$ let 
\equ{\mcl{U}^l_i=\{\pm n^i_1, \pm n^i_2,\ldots, \pm n^i_l\} \subs \mcl{U}_i,\gd:\mcl{U}^l_1\lra \mcl{U}^l_2,
\gd(n^1_j)=n^2_j,1\leq j \leq l} 
be the sets of antipodal pairs of normal vectors for the hyperplane arrangements $\{H^i_1,H^i_2\ldots,H^i_l\},i=1,2$
respectively with the bijection $\gd$ being identity on subscripts. Let $n^i_l$ be the outward normal of $H^i_l$.
Consider the following system of equations for $i=1,2$.
\equa{n_l^i&=\us{k=1}{\os{m-2}{\sum}}x^k_in^i_{j_k}+ x^{m-1}_in^i_{j_{m-1}} + x^{m}_in^i_{j_{m}}\\
n_l^i&=\us{k=1}{\os{m-2}{\sum}}y^k_in^i_{j_k}+ y^{m}_in^i_{j_{m}} + y^{m+1}_in^i_{j_{m+1}}\\
n_l^i&=\us{k=1}{\os{m-2}{\sum}}z^k_in^i_{j_k}+ z^{m-1}_in^i_{j_{m-1}} + z^{m+1}_in^i_{j_{m+1}}}
Now we have for \equ{1\leq k \leq m-2, sign(x^k_1)=sign(x^k_2),sign(y^k_1)=sign(y^k_2),sign(z^k_1)=sign(z^k_2)} 
and also signs are the same for the remaining respective coefficients as well for $i=1,2$.

For $i=1,2$ let $P^i_{\{j_1,j_2,\ldots,j_{m-2}\}}$ be the orthogonal projection onto the two-dimensional space $V^i$ with bases given by 
\equ{\{n^i_A,n^i_B\}\text{ or }\{n^i_B,n^i_C\} \text{ or }\{n^i_A,n^i_C\}} and whose kernel $U^i$ is spanned by $\{n^i_{j_1},n^i_{j_2},\ldots,n^i_{j_{m-2}}\}$. We have
\equ{\mbb{R}^m=U^i\oplus V^i,i=1,2, P^i_{\{j_1,j_2,\ldots,j_{m-2}\}}:\mbb{R}^m \twoheadrightarrow V^i }
We have the following projected equations of vectors.
\equan{Three}{&P^i_{\{j_1,j_2,\ldots,j_{m-2}\}}(n_l^i)=\\
&x^{m-1}_i P^i_{\{j_1,j_2,\ldots,j_{m-2}\}}(n^i_{j_{m-1}}) + 
x^{m}_iP^i_{\{j_1,j_2,\ldots,j_{m-2}\}}(n^i_{j_{m}})=\\
&y^{m}_i P^i_{\{j_1,j_2,\ldots,j_{m-2}\}}(n^i_{j_{m}}) + 
y^{m+1}_iP^i_{\{j_1,j_2,\ldots,j_{m-2}\}}(n^i_{j_{m+1}})=\\
&z^{m-1}_i P^i_{\{j_1,j_2,\ldots,j_{m-2}\}}(n^i_{j_{m-1}}) + 
z^{m+1}_iP^i_{\{j_1,j_2,\ldots,j_{m-2}\}}(n^i_{j_{m+1}})}
which have the coefficients with the same sign for $i=1,2$. 
We refer to Figure~\ref{fig:Two} for an illustrative example.
Here we observe the following orthogonality conditions (perpendicularity conditions in Figure~\ref{fig:Two}).
\equa{n^i_A.n^i_{j_{m-1}}=0 &\Ra n^i_A.P^i_{\{j_1,j_2,\ldots,j_{m-2}\}}(n^i_{j_{m-1}})=0,\\
n^i_C.n^i_{j_{m}}=0 &\Ra n^i_C.P^i_{\{j_1,j_2,\ldots,j_{m-2}\}}(n^i_{j_{m}})=0,\\
n^i_B.n^i_{j_{m+1}}=0 &\Ra n^i_B.P^i_{\{j_1,j_2,\ldots,j_{m-2}\}}(n^i_{j_{m+1}})=0}

\begin{figure}[h]
\centering
\includegraphics[width = 1.0\textwidth]{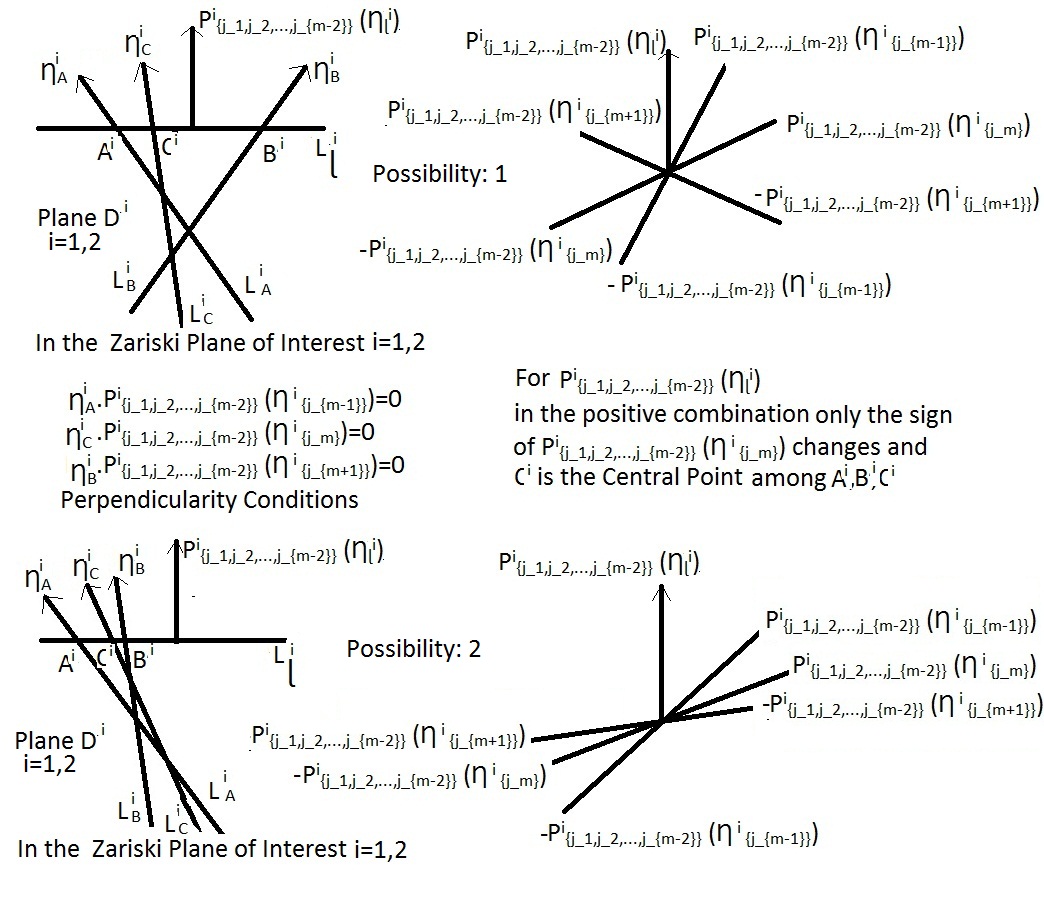}
\caption{Two-dimensional zariski plane of interest $\{j_1<j_2<\ldots<j_{m-2}\}$}
\label{fig:Two}
\end{figure}

In the right hand side of three equations~\ref{Eq:Three} each of the vectors $n^i_{j_{m-1}},n^i_{j_m},n^i_{j_{m+1}}$ appear twice.
Now we write the three equations individually with a suitable choice of signs $(sign)_t\in \{\pm 1\},1\leq t\leq 6$ with 
\equa{&\mid x^{m-1}_i\mid = (sign)_1x^{m-1}_i,\mid x^{m}_i\mid=(sign)_2x^{m}_i,\mid y^{m}_i \mid=(sign)_3y^{m}_i,\\
&\mid y^{m+1}_i \mid=(sign)_4y^{m+1}_i,\mid z^{m-1}_i \mid=(sign)_5z^{m-1}_i,\mid z^{m+1}_i \mid=(sign)_6z^{m+1}_i} for $n^i_{j_{m-1}},n^i_{j_m},n^i_{j_{m+1}}$
such that the coefficients are positive, that is,  
\equan{ThreeMod}{&P^i_{\{j_1,j_2,\ldots,j_{m-2}\}}(n_l^i)=\\&\mid x^{m-1}_i\mid  P^i_{\{j_1,j_2,\ldots,j_{m-2}\}}((sign)_1 n^i_{j_{m-1}}) + 
\mid x^{m}_i\mid P^i_{\{j_1,j_2,\ldots,j_{m-2}\}}((sign)_2 n^i_{j_{m}})=\\
&\mid y^{m}_i \mid P^i_{\{j_1,j_2,\ldots,j_{m-2}\}}((sign)_3 n^i_{j_{m}}) + 
\mid y^{m+1}_i \mid P^i_{\{j_1,j_2,\ldots,j_{m-2}\}}((sign)_4 n^i_{j_{m+1}})=\\
&\mid z^{m-1}_i \mid P^i_{\{j_1,j_2,\ldots,j_{m-2}\}}((sign)_5 n^i_{j_{m-1}}) + 
\mid z^{m+1}_i \mid P^i_{\{j_1,j_2,\ldots,j_{m-2}\}}((sign)_6 n^i_{j_{m+1}})\\}

We look for which of these three vectors $n^i_{j_{m-1}},n^i_{j_m},n^i_{j_{m+1}}$ changes its sign in $\{(sign)_1\lra (sign)_5\}$, $\{(sign)_2\lra (sign)_3\},\{(sign)_4\lra (sign)_6\}$  while 
appearing twice in these three equations~\ref{Eq:ThreeMod}. Here more importantly since $\gd$ is a convex positive bijection consistency is maintained for $i=1,2$,
that is, we have the same suitable choice of signs of vectors for the three equations as they correspond
bijectively by $\gd$.

We have reduced to the two-dimensional scenario in the plane of interest $H^i_{j_1}\cap H^i_{j_2}\cap \ldots \cap H^i_{j_{m-2}}(=D^i \text{ in Figure~\ref{fig:Two}})$
for $i=1,2$. Here we use Observation~\ref{obs:ChangeSign} to conclude that only the sign of 
the vector corresponding to central point changes when $i=1$ and when $i=2$  
These corresponding subscripts agree for $i=1,2$. More elaborately for $i=1,2$, the lines are given by 
\equa{L^i_A=L^i_{\{j_1,j_2,\ldots,j_{m-2},j_{m-1}\}}&=H^i_{j_1}\cap H^i_{j_2}\cap \ldots \cap H^i_{j_{m-2}}\cap H^i_{j_{m-1}}\\
L^i_C=L^i_{\{j_1,j_2,\ldots,j_{m-2},j_{m}\}}&=H^i_{j_1}\cap H^i_{j_2}\cap \ldots \cap H^i_{j_{m-2}}\cap H^i_{j_{m}}\\
L^i_B=L^i_{\{j_1,j_2,\ldots,j_{m-2},j_{m+1}\}}&=H^i_{j_1}\cap H^i_{j_2}\cap \ldots \cap H^i_{j_{m-2}}\cap H^i_{j_{m+1}}}
and the points of intersection of these three lines are given by
\equa{& H^i_{j_1}\cap H^i_{j_2}\cap \ldots \cap H^i_{j_{m-2}}\cap H^i_{j_{m-1}}\cap H^i_{j_{m}}\\
& H^i_{j_1}\cap H^i_{j_2}\cap \ldots \cap H^i_{j_{m-2}}\cap H^i_{j_{m-1}}\cap H^i_{j_{m+1}}\\
& H^i_{j_1}\cap H^i_{j_2}\cap \ldots \cap H^i_{j_{m-2}}\cap H^i_{j_{m}}\cap H^i_{j_{m+1}}}

and with the same notation as in the claim, the line at infinity is
\equ{L^i_{\{j_1,j_2,\ldots,j_{m-2},l\}}=H^i_{j_1}\cap H^i_{j_2}\cap \ldots \cap H^i_{j_{m-2}}\cap H^i_l (=L^i_l \text{ in Figure~\ref{fig:Two}})}
which contains the points 
\equa{P^i_{A\cup\{l\}}&=H^i_{j_1}\cap H^i_{j_2}\cap \ldots \cap H^i_{j_{m-2}}\cap H^i_{j_{m-1}}\cap H^i_l (=A^i\text{ in Figure~\ref{fig:Two}})\\
P^i_{C\cup\{l\}}&=H^i_{j_1}\cap H^i_{j_2}\cap \ldots \cap H^i_{j_{m-2}}\cap H^i_{j_{m}}\cap H^i_l (=C^i \text{ in Figure~\ref{fig:Two}})\\
P^i_{B\cup\{l\}}&=H^i_{j_1}\cap H^i_{j_2}\cap \ldots \cap H^i_{j_{m-2}}\cap H^i_{j_{m+1}}\cap H^i_l (=B^i \text{ in Figure~\ref{fig:Two}})}
in the plane of interest
\equ{H^i_{j_1}\cap H^i_{j_2}\cap \ldots \cap H^i_{j_{m-2}} (=D^i \text{ in Figure~\ref{fig:Two}}),i=1,2.}
We observe that 
\equa{&P^1_{C\cup\{l\}} \text{ is in between } P^1_{A\cup\{l\}} \text{ and }P^1_{B\cup\{l\}}\\ 
&\Llra P^2_{C\cup\{l\}} \text{ is in between } P^2_{A\cup\{l\}} \text{ and }P^2_{B\cup\{l\}}.}
We also have that $n^i_C=a_in^i_A+b_in^i_B$ with $a_i>0,b_i>0$ if and only if $P^i_{C\cup\{l\}}$ is in between 
$P^i_{A\cup\{l\}}$ and $P^i_{B\cup\{l\}}$.
Hence the claim follows. 
\end{proof}
This claim also proves that there is an isomorphism between the codimension-one arrangements on the hyperplanes 
at infinity $H^i_{l},i=1,2$ which is identity on the subscripts using Theorem~\ref{theorem:PCP}.
Now we use extension Theorem~\ref{theorem:ExtTheorem} to conclude the induction step.
\end{proof}

\begin{proof}[$\La$ Proof]
We prove the other way implication. Suppose there exists an isomorphism between
the infinity type hyperplane arrangements $(\mcl{H}_n^m)_1=\{H^1_1,H^1_2,\ldots,$ $H^1_n\}$,
$(\mcl{H}_n^m)_2=\{H^2_1,H^2_2,\ldots,H^2_n\}$ which is identity on the subscripts after renumbering. We now suppose after further renumbering the subscripts of both the hyperplane arrangements that, the infinity permutations (refer to Definition~\ref{defn:InfinityArrangements}) for both of them is the identity permutation.
Let $\mcl{U}_1=\{\pm v^1_1,\pm v^1_2,\ldots,\pm v^1_n\},\mcl{U}_2=\{\pm v^2_1,\pm v^2_2,\ldots,\pm v^2_n\}$ 
be the corresponding sets containing a pair of normal antipodal vectors 
then we have to prove that, there exists $\gd:\mcl{U}_1\lra \mcl{U}_2$ which is identity on the subscripts and which 
is a convex positive bijection. 

Let us choose an outward pointing normal vector $n^i_j,\ 1\leq j \leq m+1$ for 
$H^i_1,H^i_2,\ldots,$ $H^i_{m+1}$ with respect to the $m\operatorname{-}$dimensional simplex polyhedrality 
$\Gd^mH^i_1H^i_2\ldots H^i_{m+1}$ and then an outward pointing normal vector $n^i_l$ for 
$H^i_l, n\geq l \geq m+2$ on the other side of the zero-dimensional vertices of the earlier arrangement for $i=1,2$ at each stage.

We prove the following claim.
\begin{claim}
Consider for $1\leq k_1<k_2<\ldots<k_m<k_{m+1}\leq l$ a set of $m+1$ hyperplanes 
\equ{H^i_{k_1},H^i_{k_2},\ldots,H^i_{k_m},H^i_{k_{m+1}},i=1,2} and their normal vectors \equ{n^i_{k_1},n^i_{k_2},\ldots,n^i_{k_m},n^i_{k_{m+1}},i=1,2} 
respectively. Then for $1\leq j\leq m+1$ we have 
$n^1_{k_j}$ is an outward pointing normal of the simplex $\Gd^m H^1_{k_1}H^1_{k_2}\ldots H^1_{k_m}H^1_{k_{m+1}}$ if and only if 
$n^2_{k_j}$ is an outward pointing normal of the simplex $\Gd^m H^2_{k_1}H^2_{k_2}\ldots H^2_{k_m}H^2_{k_{m+1}}$. 
\end{claim}
\begin{proof}[Proof of Claim]
The proof is trivial if $k_j=j,1\leq j\leq m+1$ because of the choice of the normals. 

Assume $\{k_1,k_2,\ldots,k_{m+1}\} \neq \{1,2,\ldots,m+1\}$. Fix $1\leq j_0 \leq m+1$. Consider the points 
\equ{T^i=H^i_{k_1}\cap \ldots H^i_{k_{j_0-1}}\cap H^i_{k_{j_0+1}}\cap \ldots \cap H^i_{k_{m+1}},i=1,2.}
If $n^i_{k_{j_0}}$ points towards the point $T^i$ then it is an inward pointing normal of the simplex $\Gd^m H^i_{k_1}H^i_{k_2}\ldots H^i_{k_m}H^i_{k_{m+1}}$. Otherwise it is an 
an outward pointing normal of the simplex $\Gd^m H^i_{k_1}H^i_{k_2}\ldots H^i_{k_m}H^i_{k_{m+1}}$ for $i=1,2$.

Let \equ{S^i_r=H^i_1\cap \ldots \cap H^i_{r-1}\cap H^i_{r+1}\cap \ldots \cap H^i_{m+1}, 1\leq r\leq m+1,i=1,2.}
$\{S^i_r\mid 1\leq r\leq m+1\}$ are the vertices of the initial simplex $\Gd^m H^i_1H^i_2\ldots H^i_mH^i_{m+1}$,
$i=1,2$.

Now $l\geq m+2$ and there are at least $m+2$ hyperplanes. 
Since $\{k_1,k_2,\ldots,k_{m+1}\}\linebreak \neq \{1,2,\ldots,m+1\}$ there exists $1\leq r\leq m+1$ such that the point $S^i_r$ is not on the plane $H^i_{k_{j_0}}$
and different from $T^i$ for each $i=1,2$. We note that for $i=1,2$, existence of two such different points $T^i\neq S^i_r$, not on the plane $H^i_{k_{j_0}}$, does not hold if $\{k_1,k_2,\ldots,k_{m+1}\} = \{1,2,\ldots,m+1\}$.

Now the normals $n^i_{k_{j_0}}$ of $H^i_{k_{j_0}}$ point to the other side of $S^i_r$ by choice for $i=1,2$. Since the arrangements 
are isomorphic, $S^1_r,T^1$ are on the same side of $H^1_{k_{j_0}}$ if and only if $S^2_r,T^2$ are on the same side of $H^2_{k_{j_0}}$. 
Hence we conclude that $n^1_{k_{j_0}}$ is an outward pointing normal of the simplex $\Gd^m H^1_{k_1}H^1_{k_2}\ldots H^1_{k_m}H^1_{k_{m+1}}$ if and only if 
$n^2_{k_{j_0}}$ is an outward pointing normal of the simplex $\Gd^m H^2_{k_1}H^2_{k_2}\ldots\linebreak H^2_{k_m}H^2_{k_{m+1}}$. 
This proves the claim.
\end{proof}

Now we use the following fact. We have that for an $m\operatorname{-}$dimensional simplex $\Gd^m$ if we choose all the normals of the planes
as outward pointing say 
\equ{u_1,u_2,\ldots,u_m,u_{m+1}} then we can express for every $1\leq i\leq m+1$ 
\equ{-u_i=\us{j=1,j\neq i}{\os{j=m+1}{\sum}} \ga_ju_j \text{ with }\ga_i>0.}

Here the convex positive bijection
\equa{&\gd: \{\pm n^1_1, \pm n^1_2,\ldots, \pm n^1_{m+1}\} \lra \{\pm n^2_1, \pm n^2_2,\ldots, \pm n^2_{m+1}\},\\ 
&\gd(n^1_j)=n^2_j,\gd(-n^1_j)=-n^2_j, 1\leq j \leq m+1}
is an isomorphism between the truncated normal systems
and inductively for $l>m+1$ extends to a convex positive bijection  
\equa{&\gd: \{\pm n^1_1, \pm n^1_2,\ldots, \pm n^1_{l-1}\} \lra \{\pm n^2_1, \pm n^2_2,\ldots, \pm n^2_{l-1}\},\\ 
&\gd(n^1_j)=n^2_j,\gd(-n^1_j)=-n^2_j, 1\leq j \leq l}
This proves that there exists a convex positive bijection between $\mcl{U}_1$ and $\mcl{U}_2$ given by
\equ{\gd(n^1_j)=n^1_j,\gd(-n^1_j)=-n^2_j,1\leq j\leq n.}
There is also another one given by $-\gd$.

This completes the proof of Theorem~\ref{theorem:SNRT} and hence the proof of main Theorem~\ref{theorem:NRTMT}
of the article.
\end{proof}

\section{\bf{Graphs of compatible pairs associated to normal systems in three dimensions}}
\label{sec:GCPANSTD}
In this section we associate an invariant namely the graph of compatible pairs for a normal system in three
dimensions. Then we observe that this invariant determines a normal system in three dimensions up to an isomorphism.
First we need a few definitions.
\begin{defn}[Graph of Compatible Pairs]
\label{defn:GCP}
~\\
Let $\mcl{N}=\{L_1,L_2,\ldots L_n\}$ be a normal system in $\mbb{R}^3$. Let 
$\mcl{U}=\{\pm u_1,\pm u_2,\ldots,\pm u_n\}$ be the corresponding set of a pair of antipodal vectors on these lines of $\mcl{N}$. 
We associate a graph $G=(V,E)$ as follows. The vertex set of the graph is given by 
\equ{V=\{\{x,y\}\mid x,y\in \mcl{U}, x\neq \pm y\}.}
We say a vertex $\{x_1,y_1\}$ is compatible with another vertex $\{x_2,y_2\} \neq \{x_1,y_1\}$ if there exist 
positive constants $a>0,b>0,c>0,d>0$ such that $ax_1+by_1=cx_2+dy_2$. This automatically means that the set
\equ{\{x_1,y_1,x_2,y_2\}\subs \mcl{U}}
is maximally linearly independent.
The edge set $E$ of the graph is defined as follows. There is an edge between two vertices $v_1,v_2\in V$
if they are compatible.
\end{defn}
The following theorem is an immediate consequence from the definitions whose proof is straight forward.
\begin{theorem}
Let $\mcl{N}_1,\mcl{N}_2$ be two normal systems in $\mbb{R}^3$. Let $\mcl{U}_1,\mcl{U}_2$ be the sets of 
antipodal pairs of vectors respectively and $\gd:\mcl{N}_1 \lra \mcl{N}_2$ be
a convex positive bijection. Then the graphs $G_1,G_2$ of compatible pairs of normal systems respectively
are isomorphic by an isomorphism induced by $\gd$. Conversely if $\gd$ is a bijection between $\mcl{U}_1,\mcl{U}_2$
which preserves antipodal pairs such that $\gd$ induces an isomorphism of the graphs $G_1,G_2$ of compatible pairs 
then $\gd$ is a convex positive bijection. 
\end{theorem}

Now using the graph invariant of the normal system we give examples of two non-isomorphic normal systems 
below.
\subsection{Examples of two non-isomorphic normal systems in three dimensions}
\label{sec:ENINSTD}
In this section we give two examples of normal systems consisting of six lines in three dimensions 
which are not isomorphic to each other by showing that their graphs
of compatible pairs are not isomorphic which is again proved by showing vertices of
degree $1$ and degree $5$ exist in one but not in the other.

We consider the following sets $\mcl{U}_i,i=1,2$ of antipodal unit vectors on the two normal systems
$\mcl{N}_i,i=1,2$ respectively over the field of 
rational numbers $\mbb{Q}$ which is contained in all ordered fields.
Let \equa{&u_1=(1,0,0)=v_1,u_2=(0,1,0)=v_2,u_3=(0,0,1)=v_3,\\
&u_4=\big(\frac 13,\frac 23,\frac 23\big)=v_4, u_5=\big(\frac 19,\frac 49,\frac 89\big)=v_5,
u_6=\big(\frac 6{11},\frac 6{11},\frac 7{11}\big),v_6=\big(\frac 2{11},\frac 6{11},\frac 9{11}\big).}
Let \equ{\mcl{U}_1=\{\pm u_i\mid 1\leq i \leq 6\},\mcl{U}_2=\{\pm v_i\mid 1\leq i \leq 6\}}
We have $\mcl{U}_1\cap \mcl{U}_2=\{\pm u_1,\pm u_2,\pm u_3,\pm u_4,\pm u_5\}=\{\pm v_1,\pm v_2,\pm v_3,\pm v_4,\pm 
v_5\}$. 

Now we obtain the following $\binom{6}{4}=15$ equations for $\mcl{U}_1$.
\begin{enumerate}
\item $3u_4= u_1+2u_2+2u_3=(1,2,2)$.
\item $9u_5=u_1+4u_2+8u_3=(1,4,8)$.
\item $11u_6=6u_1+6u_2+7u_3=(6,6,7)$.
\item $12u_4=3u_1+4u_2+9u_5=(4,8,8)$.
\item $5u_1+21u_4=2u_2+22u_6=(12,14,14)$.
\item $88u_6=41u_1+20u_2+63u_5=(48,48,56)$.
\item $u_1+9u_5=4u_3+6u_4=(2,4,8)$.
\item $11u_6=3u_1+u_3+9u_4=(6,6,7)$.
\item $9u_1+9u_5=10u_3+22u_6=(12,12,24)$.
\item $9u_5=2u_2+6u_3+3u_4=(1,4,8)$.
\item $18u_4=6u_2+5u_3+11u_6=(6,12,12)$.
\item $54u_5=18u_2+41u_3+11u_6=(6,24,48)$.
\item $44u_6=13u_1+30u_4+9u_5=(24,24,28)$
\item $123u_4=26u_2+45u_5+66u_6=(41,82,82)$.
\item $13u_3+27u_4=27u_5+11u_6=(9,18,31)$.
\end{enumerate}
Also we obtain the following $\binom{6}{4}=15$ equations for $\mcl{U}_2$.
\begin{enumerate}
\item $3v_4=v_1+2v_2+2v_3=(1,2,2)$.
\item $9v_5=v_1+4v_2+8v_3=(1,4,8)$.
\item $11v_6=2v_1+6v_2+9v_3=(2,6,9)$.
\item $12v_4=3v_1+4v_2+9v_5=(4,8,8)$.
\item $27v_4=5v_1+6v_2+22v_6=(9,18,18)$.
\item $88v_6=7v_1+12v_2+81v_5=(16,48,72)$.
\item $v_1+9v_5=4v_3+6v_4=(2,4,8)$.
\item $v_1+11v_6=3v_3+9v_4=(3,6,9)$.
\item $v_1+27v_5=6v_3+22v_6=(4,12,24)$.
\item $9v_5=2v_2+6v_3+3v_4=(1,4,8)$.
\item $11v_6=2v_2+5v_3+6v_4=(2,6,9)$.
\item $18v_5=2v_2+7v_3+11v_6=(2,8,16)$.
\item $v_1+44v_6=18v_4+27v_5=(9,24,36)$.
\item $66v_6=2v_2+21v_4+45v_5=(12,6,54)$.
\item $v_3+11v_6=3v_4+9v_5=(2,6,10)$.
\end{enumerate}

We observe that the graph of compatible pairs $G_1=(V_1,E_1)$ has edges of degree $1$ and $5$
associated to $\mcl{N}_1$. For example
the degree of the vertex $\{-u_1,u_2\}$ is one in $G_1$ and the only edge with this vertex is with vertex
$\{u_4,-u_6\}$ (equation $(5)$ in the first set) and there is no vertex of degree one in the graph $G_2=(V_2,E_2)$ of compatible pairs associated 
to $\mcl{N}_2$. 
Similarly the vertex $\{-u_1,-u_5\}$ has degree $5$ in the graph $G_1$ with edges to the vertices
\equ{\{u_2,-u_4\},\{u_2,-u_6\},\{-u_4,-u_3\},\{-u_6,-u_3\},\{u_4,-u_6\}}
given by equations $(4),(6),(7),(9),(13)$ respectively.
There are no vertices of degree $5$ in the graph $G_2$.
This shows that not all normal systems of the same cardinality are isomorphic in dimension three 
unlike dimension two.
\section{\bf{Topology appendix}}
\label{sec:TA}
In this section we prove zariski connectedness (in the sense of point set topology) 
of positive dimensional skeletons of a hyperplane arrangement over infinite fields. We start with a definition.
\begin{defn}[$k\operatorname{-}$Dimensional Skeleton]
\label{defn:SK}
~\\
Let $\mbb{F}$ be an infinite field. Let $(\mcl{H}_n^m)^{\mbb{F}}=\{H_1,H_2,\ldots,H_n\}$ be a hyperplane arrangement in $\mbb{F}^m$.
For $0\leq k \leq n$, the skeleton $S_{n-k}$ of $(n-k)\operatorname{-}$dimensional planes is defined to be 
\equ{S_{n-k}=\us{1\leq i_1 < i_2 < \ldots < i_k\leq n}{\bigcup} H_{i_1}\cap H_{i_2}\cap \ldots \cap H_{i_k}.}
\end{defn}
Now we state the theorem of this section.
\begin{theorem}[Zariski Connectedness of Positive Dimensional Skeletons of the Hyperplane Arrangement]
\label{theorem:ZCHA}
~\\
Let $\mbb{F}$ be an infinite field. Let $(\mcl{H}_n^m)^{\mbb{F}}=\{H_1,H_2,\ldots,H_n\}$ with $H_i\subs \mbb{A}^m_{\mbb{F}},1\leq i\leq n$ be a hyperplane arrangement.
For $m-1\geq k\geq 1$ the skeleton $S_k$ is connected in the point set topological sense in the zariski topology on the affine space $\mbb{A}^m_{\mbb{F}}$. 
\end{theorem}
\begin{note}
Theorem~\emph{\ref{theorem:ZCHA}} is used in proving extension Theorem~\emph{\ref{theorem:ExtTheorem}}. 
\end{note}
Before we prove the theorem we mention the following.
\begin{note}
\label{note:Zariski}
\begin{enumerate}
\item The affine space $\mbb{A}^m_{\mbb{F}},m\geq 1$ is irreducible if $\mbb{F}$ is infinite.
\item A union of two intersecting $k\operatorname{-}$dimensional planes is zariski connected but not zariski irreducible.
\item The union of two skew hyperplanes arrangements $(\mcl{H}^r_{n_1})^{\mbb{F}},(\mcl{H}^r_{n_2})^{\mbb{F}}$ 
embedded (skewly) in an $r\operatorname{-}$dimensional affine space with $r\leq m-1$ is neither zariski irreducible nor zariski 
connected.
\item \label{obs:ZariskiClopen}
Let $X$ be a topological space and let $U\subs X$ be a clopen (closed and open) set. If $Y\subs X$
is connected or irreducible then we have $U\cap Y\neq \es \Ra Y\subs U$.
\end{enumerate}
\end{note}
Now we prove the theorem.
\begin{proof}[Proof of Theorem~\emph{\ref{theorem:ZCHA}}]
Let $U\subs S_1$ be a non-empty clopen set. Since $U$ is non-empty $U$ contains a zariski line as it is irreducible. 
We can change the $m-1$ subscripts of a zariski line one by one to move from
one zariski line to another inside $U$ using Note~\ref{note:Zariski}(\ref{obs:ZariskiClopen}). Hence $U=S_1$. 
Thus the one-dimensional skeleton is zariski connected.
Similarly we have that for any $1\leq i\leq m-1$, if $U\subs S_i$ is a non-empty clopen set then $U=S_i$. Thus the $i\operatorname{-}$dimensional skeleton 
is zariski connected. We use Note~\ref{note:Zariski}(\ref{obs:ZariskiClopen}) in the proof. 
This proves Theorem~\ref{theorem:ZCHA}.
\end{proof}


\begin{thebibliography}{1}
\bibitem{CPAK} C.~P.~Anil Kumar,
{\it On the triangles in certain types of line arrangements}, Discrete Mathematics, Algorithms and Applications, Online Published Nov. 2020, \url{https://doi.org/10.1142/S1793830921500312}, \url{https://arxiv.org/pdf/1906.05120.pdf}

\bibitem{MR3618796} A.~Dimca,
{\it Hyperplane Arrangements: An Introduction}, Universitext, Springer, Cham, 2017, pp. xii+200, ISBN-13: {\bf 978-3-319-56220-9}, ISBN-13: {\bf 978-3-319-56221-6}, \url{https://doi.org/10.1007/978-3-319-56221-6}, MR3618796

\bibitem{MR0780184} N.~Jacobson,
{\it Basic Algebra I}, Dover Books on Mathematics, Second Edition, 2009, ISBN-13:  978-0-486-47189-1, Unabridged republication originally published by W.~H.~Freeman and Co., San Francisco, 1985, xviii+499 pp, MR0780184

\bibitem{MR1009787} N.~Jacobson,
{\it Basic Algebra II}, Dover Books on Mathematics, Second Edition, 2009, ISBN-13:  978-0-486-47187-7, Unabridged republication originally published by W.~H.~Freeman and Co., San Francisco, 1989, xviii+686 pp, MR1009787

\bibitem{MR3702317} E.~Katz,
{\it Matroid Theory for Algebraic Geometers, In: M.~Baker, S.~Payne (eds) Nonarchimedean and tropical geometry}, pp. 435-517, Simons Symposium Proceedings, Springer, Cham, 2016, \url{https://doi.org/10.1007/978-3-319-30945-3_12}, \url{https://arxiv.org/abs/1409.3503}, MR3702317

\bibitem{MR1878556} S.~Lang,
{\it Algebra}, Third Edition, Graduate Texts in Mathematics,   211,  Springer-Verlag, New York, 2002, xvi+914 pp, ISBN-13: 978-0-387-95385-4, \url{https://doi.org/10.1007/978-1-4613-0041-0}, MR1878556

\bibitem{MR1117042} N.~Megiddo,
{\it On the complexity of linear programming}, In: Advances in economic theory: Fifth world congress,
T. Bewley, ed., Cambridge University Press, Cambridge (1987), pp. 225-268, \url{https://doi.org/10.1017/CCOL0521340446.006}, MR1117042 

\bibitem{MR1217488} P.~Orlik, H.~Terao,
{\it Arrangements of Hyperplanes}, Grundlehren der Mathematischen Wissenschaften [Fundamental Principles of Mathematical Sciences], Vol. {\bf 300} Springer-Verlag, Berlin, 1992, pp. xviii+325, ISBN-13: {\bf 978-3-540-55259-6}, \url{https://doi.org/10.1007/978-3-662-02772-1}, MR1217488

\bibitem{MR1631413} S.~Smale.
{\it Mathematical problems for the next century}, The Mathematical Intelligencer (1998), 
Vol. {\bf 20(2)}, pp 7–15, \url{https://doi.org/10.1007/BF03025291}, MR1631413

\bibitem{MR1754783} S.~Smale,
{\it Mathematical problems for the next century}, In Arnold, V. I.; Atiyah, M.; Lax, P.; Mazur, B.; 
Mathematics: frontiers and perspectives, International Mathematical Union, American Mathematical Society, (1999) pp. 271–294, 
ISBN-10 {\bf 0821820702}, MR1754783

\end{thebibliography}
\end{document}